\documentclass[11pt]{article}
\newcommand{\sfrac}[2]{{\textstyle\frac{#1}{#2}}}
\usepackage{amsmath,amssymb}
\usepackage{graphicx}  
\usepackage{tikz}

 \renewcommand{\Pr}{{\mathbb{P}}}

\newcommand{\PP}{\mathcal P}

\newcommand{\bX}{\mathbf{X}}
\newcommand{\bY}{\mathbf{Y}}

\newcommand{\by}{\mathbf{y}}
 
\newtheorem{Lemma}{Lemma}
\newtheorem{Theorem}[Lemma]{Theorem}

 \newcommand{\qed}{\ \ \rule{1ex}{1ex}}
 
 \newcommand{\Bin}{\mathrm{Bin}}

 \newcommand{\ind}{1}
\begin{document}

\title{Markov chains and mappings of distributions on compact spaces}
 \author{David J. Aldous\thanks{Department of Statistics,
 367 Evans Hall \#\  3860,
 U.C. Berkeley CA 94720;  aldous@stat.berkeley.edu;
  www.stat.berkeley.edu/users/aldous.}
    \and
  Shi Feng\thanks{Department of Mathematics, Cornell University; sf599@cornell.edu}
}

 \maketitle
 
 \begin{abstract}
 Consider a compact metric space $S$ and a pair $(j,k)$ with $k \ge 2$ and $1 \le j \le k$.
 For any probability distribution $\theta \in \PP(S)$, define a Markov chain on $S$ by:
 from state $s$, take $k$ i.i.d. ($\theta$) samples, and jump to the $j$'th closest. 
 Such a chain converges in distribution to a unique stationary distribution, say $\pi_{j,k}(\theta)$.
So this defines a mapping $\pi_{j,k}: \PP(S) \to \PP(S)$.
What happens when we iterate this mapping?
In particular, what are the fixed points of this mapping?
We present a few rigorous results, to complement our extensive simulation study elsewhere.
 \end{abstract}
 
 

\section{Introduction}

This article discusses a rather novel topic whose motivation may seem obscure, 
so we start with informal background that led to the formulation of the topic.
Write $S =  (S,d)$ for a  compact metric space..  
Then the identity function $f(s) := s$ makes sense for every $S$.
Is there any more interesting explicit function $S \to S$ whose definition makes sense for every $S$?
For example one might try 
$ f(s) := \arg \max_y d(s,y) $,
that is the most distant point from $s$; this works  for any space $S$ with the property that the most distant point is always unique, but  not for all $S$. 
Our introspection suggests that in fact there is no non-trivial such ``general" function.

Instead let us write $\PP(S)$ for the space of probability distributions on $S$, and recall that 
$\PP(S)$ is a compact metric space under the usual weak topology. 
The observation above suggests that there may be no non-trivial function 
$\PP(S) \to \PP(S)$ whose definition makes sense for every $S$. 
But this is false!
This article investigates a particular family of such functions -- the reader may care to try to invent different examples.

\paragraph{The Markov chain.}
Given $S$ and $\theta \in \PP(S)$, consider the following discrete time Markov chain on state space $S$: from 
point $s$ make the step  to the nearer of $2$ random points drawn i.i.d. from $\theta$, breaking ties uniformly at random.
This scheme naturally generalizes as follows: fix $k \ge 2$ and $1 \le j \le k$, 
and step from $s$ to the $j$'th nearest of  $k$ random points drawn i.i.d. from $\theta$, again breaking ties uniformly at random. 
This defines a kernel $K^{\theta,j,k}$ on $S$.
Write the associated chain as $\bX^{\theta,j,k} = (X^{\theta,j,k}(t), t = 0, 1, 2, \ldots)$.
Theorem  \ref{T:1}  proves that this chain always has a unique stationary distribution,  which we can call $\pi_{j,k}(\theta)$.
So now we have defined a mapping $\pi_{j,k}: \PP(S) \to \PP(S)$ for every $S$.
Theorem  \ref{T:1} also  proves that the distributions $\theta$ and $\pi_{j,k}(\theta)$ are mutually absolutely continuous, so in particular
have the same support.

\paragraph{Invariant measures for the mappings.}
These mappings $\pi_{j,k}$ have apparently not been studied previously, even for special spaces $S$ 
 and the simplest case $k = 2$.
Amongst the  range of questions one could ask, 
we will seek to study it as a dynamical system.
Given a mapping $\pi$ from a space to itself, it is mathematically natural to consider iterates
\begin{equation}
 \pi^{n+1}(\theta) = \pi(\pi^n(\theta)), n \ge 1. 
\label{iterates}
\end{equation}
In our setting it seems plausible that (at least for typical initial $\theta$) the iterates should converge to some limit, that is we expect weak convergence
\begin{equation}
 \pi^n_{j,k}(\theta) \to_w \phi \mbox{ as } n \to \infty
 \label{pin}
 \end{equation}
and then we expect\footnote{As observed in \cite{mappings_long}, the Markov chain is not always a Feller process, so 
\eqref{phiFP} does not immediately follow from \eqref{pin}.}
 the limit $\phi$ to satisfy the {\em fixed point} or  {\em invariant distribution} condition
\begin{equation}
\pi_{j,k}(\phi) = \phi . 
\label{phiFP}
\end{equation}
Some comments about this set-up.

\noindent
(a) The {\em iterative procedure} \eqref{iterates} does not have any simple stochastic process interpretation, in contrast to the mapping $\theta \to \pi_{j,k}(\theta)$ derived from the Markov chain.

\noindent
(b) The equation for the stationary distribution $ \pi_{j,k}(\theta)$, which for finite $S$ is the elementary Markov chain 
relation $\pi = \pi \mathbf{P}$, is a linear equation, whereas the fixed point equation \eqref{phiFP} is 
decidedly non-linear. 

\noindent
(c) On any $S$ and for any $(j,k)$, two types of measures are always invariant:
we call these the {\em omnipresent} measures.

\begin{itemize}
\item The distribution $\delta_s$ degenerate at one point $s$;
\item The uniform two-point distribution 
$\delta_{s_1,s_2} = \frac{1}{2}(\delta_{s_1} + \delta_{s_2})$.
\end{itemize}

\noindent
(d) If an invariant distribution on $S$ has support $S_0 \subset S$ then we can regard it as 
an invariant distribution on $S_0$. 
So the essential question is:  given $S$, what are the invariant distributions {\em with full support}?
Note that when $ \pi^n_{j,k}(\theta) \to_w \phi $ the distributions  $( \pi^n_{j,k}(\theta), n \ge 1)$ 
all have the same support (by Theorem \ref{T:1}) but $\phi$ may (as usually has, it turns out) 
have smaller support.

\noindent
(e) If $\theta \in \PP(S)$ is invariant under an isometry $\iota$ of $S$ then $\pi_{j,k}(\theta)$ is also invariant under $\iota$.
So we expect a limit of the iterative procedure \eqref{iterates} to also be invariant under $\iota$.
We say these invariant measures ``exist by symmetry", for instance  the
 Haar probability measure on a compact group S with a metric invariant under the group action.

\paragraph{Motivation.}
There is no notion of ``uniform distribution" applicable to every compact metric space $S$.
The original motivation for this project was the hope that our invariant distributions might 
provide a proxy for uniform distributions on a general $S$.
We attempted to find such distributions  via numerically implementing the iterative procedure 
on various spaces $S$.
What we found was that, in the absence of some special symmetry property preserved
under the  iterative procedure, one almost always obtained a limit supported on only one or two points, the {\em omnipresent} measures mentioned above.
This seemed counter-intuitive, and prompted the further study of 
invariant measures, even though the
original motivation turned out to be unsuccessful.

.
\paragraph{What numerics and simulation suggest.}
Our quite extensive study via numerics and simulation is described in a  companion document 
\cite{mappings_long}, and suggests the following big picture.

\noindent
(a) For $k = 2$, there are no invariant measures other than the omnipresent ones, except perhaps for
``exist by symmetry" ones; 
with that exception, for $j = 1, k = 2$ the iterates \eqref{iterates} converge to some $\delta_s$,  
and for $j = 2, k = 2$ the iterates \eqref{iterates} converge to some $\delta_{s_1,s_2}$. 
The precise limits $(s, s_1, s_2)$ may depend on the initial $\theta$.  In the case of 
$\delta_{s_1,s_2}$, the pair $(s_1,s_2)$ is a local maximum of $d(\cdot, \cdot)$.

\noindent
(b) For larger $k$, for some types of space $S$ there are additional {\em sporadic} invariant measures; 
we don't see a pattern.

\noindent
(c) For large $k$, as $j$ increases we see a transition, around $j/k = 0.7$, between convergence to some $\delta_s$ 
and convergence to some $\delta_{s_1,s_2}$.
However there seems no reason to believe that there is a universal value near $0.7$.

\noindent
(d) Except for the omnipresent ones, all invariant measures $\phi$ that we have encountered 
are {\em unstable}, in that from any initial distribution that is $\phi$ plus a generic (not symmetry-preserving) 
small perturbation, the iterates converge to some $\delta_s$ or  $\delta_{s_1,s_2}$.

\paragraph{What can we actually prove?}
In short: very little.
Here are the results that we will derive in this article.

\begin{itemize}
\item Theorem  \ref{T:1} is the Markov chain convergence result.
\item Results in section \ref{sec:2or3} for $|S| = 2$ or $3$  are consistent with general picture  above.
\item Theorem \ref{P:k=2}: 
{\em For every  $S$, the set of invariant distributions for $\pi_{1,2}$ is the same as the set of invariant distributions for $\pi_{2,2}$.}  
This is surprising, in that  apparently (as in  (a) above) the iterates almost always converge to some $\delta_s$ for $\pi_{1,2}$,
but to some $\delta_{s_1,s_2}$ for $\pi_{2,2}$.
\item Theorem \ref{P:0122}:
 {\em There are no $\pi_{1,2}$ or $\pi_{2,2}$-invariant distributions on 
 the interval $[0,1]$ other than the omnipresent ones.}
 \item Theorem \ref{Conj}:
  {\em There are no $\pi_{1,2}$ or $\pi_{2,2}$-invariant distributions on a space of
finite binary tree leaves 
other than the omnipresent ones.}
\end{itemize}

Of course,  for any specific $S$, one can simply write out the {\em fixed point} definition \eqref{phiFP} and seek some ad hoc 
method of finding all solutions. 
The results above carry this through (for $\pi_{1,2}$) for $|S| = 3$ and for the interval $[0,1]$, 
and for  leaf-labeled binary trees. 
But these are essentially
``proofs by contradiction" using specific features of the specific class of spaces.
For general $S$ and $\pi_{1,2}$ one feels there should be some ``contraction" argument for 
the iterates $\pi^n_{1,2}(\theta)$ -- the distributions should become more concentrated as $n$ increases -- 
but we are unable to formalize that general idea.

\section{Existence and uniqueness of stationary distributions}

\begin{Theorem}
\label{T:1}
Consider a compact metric space $(S,d)$ and a probability distribution $\theta \in \PP(S)$.
For each pair $1 \le j \le k, \ k \ge 2$, the Markov chain
$\bX^{\theta,j,k} = (X^{\theta,j,k}(t), t = 0, 1, 2, \ldots)$ 
has a unique stationary distribution $\pi_{j,k}(\theta)$.
From any initial point, the variation distance $D(t)$ between $\pi_{j,k}(\theta)$ and the distribution of $X^{\theta,j,k}(t)$ satisfies
\begin{equation}
 D(2t) \le (1 - 1/k^{k-1})^t, \quad 1 \le t < \infty 
 \label{VD}
 \end{equation}
and so there is convergence to stationarity in variation distance. 
Moreover, for $\pi = \pi_{j,k}(\theta)$ 
\begin{equation}
\theta^k(A) \le \pi(A) \le k \theta(A) , \ A \subseteq S 
\label{tkA}
\end{equation}
and so $\pi$ and $\theta$ are mutually absolutely continuous.
\end{Theorem}
Note that the bound on variation distance depends only on $k$.

\noindent
\begin{proof}
First note that for any partition $(B_i, 1 \le i \le k)$ of $S$ we have
\begin{equation}
\sum_i (\theta(B_i))^k \ge 1/k^{k-1} 
\label{kk}
\end{equation}
because by convexity the sum is minimized when $\theta(B_i) \equiv 1/k$.

We construct the process $X(t) = X^{\theta,j,k}(t)$ in the natural way, by 
creating i.i.d. $\theta$-distributed $(\bY(t) = (Y_i(t), 1 \le i \le k), t \ge 1)$ and defining for $t \ge 1$
\begin{quote}
$X(t)$ is the element of  $(Y_i(t), 1 \le i \le k)$ attaining the
$j$'th smallest value of $(d(X(t-1),Y_i(t)), 1 \le i \le k)$.
\end{quote}
In defining the re-ordering to determine ``$j$'th smallest", we break ties in accordance with the original $i$ -- that is, if 
$d(X(t-1),Y_{i_1}(t)) = d(X(t-1),Y_{i_2}(t)) $ for $i_1 < i_2$ then we put the $i_1$ term before the $i_2$ term in the reordering.
Because the $Y_i$ are i.i.d. this has the same effect as breaking the tie randomly.

We define the natural coupling $(X(t), X^\prime(t))$ of two chains started from arbitrary different states by using the same
realizations of $Y_i(t)$ for each chain.
We first seek to upper bound the coupling time 
$ T := \min\{t: X(t) = X^\prime(t)\}$.
Consider a realization $\by = (y_i, 1 \le i \le k)$ of $\bY(t+2)$.
This $\by$ induces a partition of $S$, say  $(B_i(\by), 1 \le i \le k)$, where 
$B_i(\by)$ is the set of $s \in S$ such that $d(s,y_i)$ is the $j$'th smallest of $(d(s,y_u), 1 \le u \le k)$, breaking ties as above.
The central part of the proof is the observation that the event $\{T \le t+2\}$ includes the event 
\begin{equation}
 \mbox{each component of $\bY(t+1) $ is in the same set of the partition $(B_i(\bY(t+2)), 1 \le i \le k)$.} 
 \label{eachc}
 \end{equation}
Now $\bY(t+1) $ is independent of $\bY(t+2) $, so we can apply (\ref{kk})  to show that event (\ref{eachc}) has probability $ \ge 1/k^{k-1} $.
This remains true conditional on $(X(t), X^\prime(t))$, 
and hence conditional on $\{X(t) \neq X^\prime(t)\}$,
implying that 
\[
\Pr(T \le t+2 | T > t) \ge 1/k^{k-1} .
\]
So inductively
\begin{equation}
\Pr(T > 2t)   \le (1 - 1/k^{k-1})^t, \quad 1 \le t < \infty .
\label{TTt}
\end{equation}
This is true for arbitrary initial distributions $\theta$ and  $\theta^\prime \in \PP(S)$, 
and so in particular for $\theta$ and $\theta^{(2)}$, 
where $\theta^{(t)}$ denotes the distribution of $X^{\theta,j,k}(t)$.
So (\ref{TTt}) bounds the variation distance 
\[ || \theta^{(2t+2)} - \theta^{(2t)} ||_{VD} \le (1 - 1/k^{k-1})^t, \quad 1 \le t < \infty 
\]
and similarly
\[ || \theta^{(2t+1)} - \theta^{(2t)} ||_{VD} \le (1 - 1/k^{k-1})^t, \quad 1 \le t < \infty .
\]
Now variation distance is a complete metric on $\PP(S)$, so $\theta^{(t)}$ converges in variation distance to a limit $\pi$,
and $\pi$ is a stationary distribution for the kernel $K^{\theta,j,k}$.
Then applying (\ref{TTt}) to $\pi$ and an arbitrary other initial distribution establishes (\ref{VD}) and shows that $\pi$ is the
{\em unique} stationary distribution.
Then  (\ref{tkA}) follows by considering the first step $(X(0),X(1))$ of the stationary chain, because for $A \subset S$
\[    \cap_i \{ Y_i(1) \in A \}   \subseteq        \{ X(1) \in A \} \subseteq \cup_i \{ Y_i(1) \in A \} .\]
\qed
\end{proof}

{\bf Remarks.}
The variation distance bound (\ref{VD}) is exponentially decreasing in time, but it is more natural to consider {\em mixing time} in the sense of  \cite{MCMC}. 
The example of the uniform distribution $\theta$ on a  2-point space with $j = 1$ shows that the mixing time as a function of $k$ 
can be order $2^k$.

The proof of Theorem \ref{T:1} does not say anything about $\pi_{j,k}(\theta)$ except (\ref{tkA}).
We do not know if there are  informative analytic descriptions of $\pi_{j,k}(\theta)$  in terms of $\theta$.

\section{Two or 3 points}
\label{sec:2or3}
\subsection{Two points -- the binomial case} 
\label{sec:binomial}

The case of a 2-element space $S = \{a, b\}$ and general $(j,k)$ is not completely trivial.
Here is an outline -- for more details see \cite{mappings_long}.

Parametrizing a distribution $\theta$ on $S$ by $p := \theta(a)$, we view the mapping $\pi_{j,k}: \PP(S) \to \PP(S)$  as a mapping $\pi_{j,k}: [0,1] \to [0,1]$.
From the stationary distribution we find, in terms of binomial variables, 
\[
\pi_{j,k}(p) = \frac{\Pr(\Bin(k,p) > k - j) } {\Pr(\Bin(k,p) > k - j) + \Pr(\Bin(k,p) < j)   }  .
\]
So a fixed point is a solution of the equation
\begin{equation}
\pi_{j,k}(p) = p .
\label{Bin:FP}
\end{equation}
The omnipresent fixed points are $p = 0, p = 1/2, p =1$; are there others?
By symmetry it is enough to consider $0 < p < 1/2$.

For  given $(j,k)$, we observe in \cite{mappings_long} three possible types of qualitative behavior: 

\noindent
(i) $ \pi^n_{j,k}(p) \to 0 \mbox{ as } n \to \infty, \mbox{ for all }  0 < p < 1/2$.

\noindent
(ii)  $ \pi^n_{j,k}(p) \to 1/2 \mbox{ as } n \to \infty, \mbox{ for all }  0 < p < 1/2$.

\noindent
(iii)  There exists a critical value $p_{crit} \in (0,1/2)$ which is unstable: that is, 

$p_{crit}$ is invariant 

and $ \pi^n_{j,k}(p) \to 0 \mbox{ as } n \to \infty, \mbox{ for all }  0 < p < p_{crit}$ 

and  $ \pi^n_{j,k}(p) \to 1/2 \mbox{ as } n \to \infty, \mbox{ for all }  p_{crit}  < p < 1/2$.

\noindent
Case (iii) first arises with $k = 5, j = 4$, and persists for larger values of $k$.
For instance, with $k = 8$ we observe case (i) for $1 \le j \le 5$, case (iii) for $j = 6$ with $p_{crit} = 0.26405$, 
and case (ii) for $j = 7,8$.

Of course the 2-point space is very special.
The occurrence of these ``sporadic" case (iii) fixed points seems much rarer in other spaces.

\subsection{Three elements} 
\label{sec:trinomial}

Here we consider $S = \{a,b,c\}$ where the three distances are distinct, say 
\begin{equation}
d(a,b) < d(a,c) < d(b,c) .
\label{abc}
\end{equation}
\begin{Theorem}
\label{T:3}
If $S$ satisfies \eqref{abc} then there is no $\pi_{1,2}$-invariant distribution except the omnipresent ones.
\end{Theorem}
\begin{proof}
It is enough to prove that there is no  invariant distribution 
$\theta = (\theta_a, \theta_b, \theta_c)$ 
with each term strictly positive.
So suppose, to get a contradiction, such $\theta$ exists.

Take $Y, Y_1, Y_2$ independent with  distribution $\theta$.
Invariance says that the random variable $X$ defined as
\begin{eqnarray*}
 X &=& Y_1 \mbox{ if } d(Y,Y_1) < d(Y,Y_2) \\
    &=& Y_2 \mbox{ if } d(Y, Y_2) < d(Y, Y_1) 
\end{eqnarray*}    
will also have distribution $\theta$.
Writing out  the ways that $X$ can be $c$ or $b$ or $a$ gives the equations
\begin{eqnarray*}
\theta_c &=& \theta_c(1-(1-\theta_c)^2) + \theta_b\theta_c^2 + \theta_a\theta_c^2 \\
\theta_b & = & \theta_c\theta_b^2 + \theta_b(1-(1-\theta_b)^2) + \theta_a (\theta_b^2 + 2\theta_b\theta_c) \\
\theta_a &=& \theta_c(\theta_a^2+2\theta_a\theta_b) + \theta_b(\theta_a^2 + 2\theta_a\theta_c) + \theta_a(1 - (1-\theta_a)^2).
\end{eqnarray*}
Because each term of $\theta$ is strictly positive, we can
cancel the common terms to get
\begin{eqnarray}
1 &=& 1 - (1-\theta_c )^2 + (1-\theta_c )\theta_c   \label{cc}\\
1 & = & \theta_c  \theta_b + (1-(1-\theta_b )^2) + \theta_a  (\theta_b  + 2\theta_c ) \label{bb}\\
1 &=& \theta_c ( \theta_a +2\theta_b ) + \theta_b ( \theta_a  + 2\theta_c ) + 1 - (1-\theta_a )^2) \nonumber .
\end{eqnarray}
Equation \eqref{cc} reduces to $2\theta_c^2 -3\theta_c + 1 = 0$ with solutions $\theta_c = 1$ or $1/2$.
The solution with  $\theta_c = 1$ is excluded by supposition, so we must have $\theta_c = 1/2$.
Now we have
$\theta_a = 1/2 - \theta_b$;   inserting into \eqref{bb}, the equation
 reduces to $2\theta_b^2 - 2\theta_b + \frac{1}{2} = 0$ with solution $\theta_b = \frac{1}{2}$.
So $\theta_a = 0$, contradicting the supposition.
\qed
\end{proof}

Theorem  \ref{Conj} establishes a more general result, but we have given the simpler proof above to demonstrate 
the style of ``proof by contradiction" to be used later.

\section{The $k = 2$ case.}

\begin{Theorem}
\label{P:k=2}
For every compact metric space $S$, the set of invariant distributions for $\pi_{1,2}$ is the same as the set of invariant distributions for $\pi_{2,2}$.
\end{Theorem}
\begin{proof}
Given $\theta \in \PP(S)$, the transition kernel $K = K^{\theta,1,2}$ for $\pi_{1,2}$ can be written as a Radon-Nikodym density w.r.t. $\theta$ as follows.
\begin{eqnarray*}
\frac{dK(x,\cdot)}{d \theta(\cdot)} (y) &=& 2 \theta \{z: d(x,z) > d(x,y) \} +  \theta \{z: d(x,z) = d(x,y) \}  \\
&=& \int \left( 2 \cdot \ind_{\{z: d(x,z) > d(x,y) \}} + \ind_{\{z: d(x,z) = d(x,y) \}} \right) \theta(dz) .
\end{eqnarray*}
So the identity $\theta = \theta K$ characterizing a $\pi_{1,2}$-invariant distribution $\theta$ can be written in density form as
\begin{eqnarray}
1 &=& \int \theta(dx) \   \frac{dK(x,\cdot)}{d \theta(\cdot)} (y) \nonumber \\
   &=& \int \int  \left( 2 \cdot \ind_{\{z: d(x,z) > d(x,y) \}} + \ind_{\{z: d(x,z) = d(x,y) \}} \right) \theta(dz) \theta(dx)  \label{1222}
\end{eqnarray}
where the equality holds for $\theta$-a.a. $y$.  
Because
\[
1 = \int \int 1 \ \ \theta(dz) \theta(dx)  
= \int \int 
\left(  \ind_{\{z: d(x,z) > d(x,y) \}} + \ind_{\{z: d(x,z) = d(x,y) \}}   +  \ind_{\{z: d(x,z) < d(x,y) \}}   \right)
 \ \ \theta(dz) \theta(dx)  
\]
we have from (\ref{1222}) that
\begin{equation}
 \int \int    \ind_{\{z: d(x,z) > d(x,y) \}} \ \  \theta(dz) \theta(dx)  =    \int \int   \ind_{\{z: d(x,z) < d(x,y) \}} \ \ \  \theta(dz) \theta(dx)  .
 \label{1223}
 \end{equation}
Analogous to (\ref{1222}), the identity characterizing a $\pi_{2,2}$-invariant distribution $\phi$ can be written as
\begin{equation}
1 =  \int \int  \left( 2 \cdot \ind_{\{z: d(x,z) < d(x,y) \}} + \ind_{\{z: d(x,z) = d(x,y) \}} \right) \phi(dz) \phi(dx)  . \label{1224} 
 \end{equation}
By (\ref{1222})  and (\ref{1223}), any $\pi_{1,2}$-invariant distribution $\theta$ satisfies (\ref{1224})  and is therefore a $\pi_{2,2}$-invariant distribution.
The converse holds via the analog of (\ref{1223}) for $\phi$.
\qed
\end{proof}

 \section{The case $S = [0,1]$}
 Numerical study in \cite{mappings_long} suggests that there are no invariant distributions on $[0,1]$ 
 with full support, for any $(j,k)$. 
 Theorem \ref{P:0122} proves a slightly stronger result in the case $k=2$ 
 (recall that by Theorem \ref{P:k=2} the cases $j = 1$ and $j = 2$ here are identical).
 The stronger form is not true for general $(j,k)$, for instance 
 the uniform distribution on the 4 points $\{0,0.4,0.6,1\}$ is
invariant for $\pi_{3,4}$.

\begin{Theorem}
\label{P:0122}
 There are no $\pi_{2,2}$-invariant distributions on $[0,1]$ other than those of the form $\delta_{s}$ or $\delta_{s_1,s_2}$.
\end{Theorem}
By considering the endpoints of the support of an invariant distribution, and scaling, this reduces to proving
\begin{quote}
{\bf equivalent assertion:} 
The only $\pi_{2,2}$-invariant distribution on $[0,1]$ whose support contains both $0$ and $1$ is the distribution $\delta_{0,1}$.
\end{quote}
We will prove this in two steps.
\begin{Lemma}
\label{L:0122}
There is no $\pi_{2,2}$-invariant distribution whose support contains  $0$ and which assigns zero weight to the point $0$.
\end{Lemma}
\begin{proof}
For a proof by contradiction,
suppose such an invariant distribution $\theta$ exists.
Take $Y, Y_1, Y_2$ independent with  distribution $\theta$.
Invariance says that the random variable $X$ defined as
\begin{eqnarray*}
 X &=& Y_2 \mbox{ if } |Y- Y_1| < |Y- Y_2| \\
    &=& Y_1 \mbox{ if } |Y- Y_2| < |Y- Y_1| 
\end{eqnarray*}    
(with our usual convention about ties)
will also have distribution $\theta$.
Fix $0<x<1$.
From the definition we have the inclusion of events
\[ \{X \le x\} \subseteq A_1 \cup A_2 \cup A_3 \] 
where
\begin{eqnarray*}
A_1 & := & \{Y_1 \le x, \ Y_2 \le x\} \\
A_2 & := & \{Y_1 \le x, \ Y_2 > x, \ Y \ge \sfrac{1}{2}(Y_1+Y_2) \} \\
A_3 & := & \{Y_2 \le x, \ Y_1 > x, \ Y \ge \sfrac{1}{2}(Y_1+Y_2) \} 
\end{eqnarray*} 
and the $(A_i)$ are disjoint.
Now note that
\[ A_2 \subseteq  \{Y_1 \le x, \ Y \le  \sfrac{1}{2}(x+Y_2) \} \]
and similarly for $A_3$.
So by independence, the distribution function $F$ of $\theta$ satisfies
\begin{equation}
 F(x) \le F^2(x) + 2 F(x) \Pr(Y \le  \sfrac{1}{2}(x+Y_2)) .
 \label{F2F}
 \end{equation}
By hypothesis, $F(x) > 0 $ for $x > 0$ and $F(x) \downarrow 0$ as $x \downarrow 0$.
So we can divide both sides of (\ref{F2F}) by $F(x)$ and take limits as $x \downarrow 0$ and deduce
\[ \Pr(Y \le  \sfrac{1}{2}Y_2)  \ge \sfrac{1}{2} . \]
By symmetry we also have
$\Pr(Y_2 \le  \sfrac{1}{2}Y)  \ge \sfrac{1}{2} $,
and so 
\[ \Pr(\sfrac{1}{2}Y < Y_2 < 2 Y) = 0 . \]
But this is impossible for i.i.d. samples from a distribution $\theta$ on $(0,1]$,  because it would remain true
for $\theta$ conditioned on an interval of the form $[y,3y/2]$.
\qed
\end{proof}

Using Lemma \ref{L:0122} and reflection-symmetry of $[0,1]$, to prove the {\em equivalent assertion} and hence Theorem  \ref{P:0122} 
it will be sufficient to prove
\begin{Lemma}
 \label{L:0123} 
If $\theta$ is a $\pi_{2,2}$-invariant distribution and $\theta_0 > 0, \theta_1 > 0$ then $\theta = \delta_{0,1}$.
\end{Lemma}
Here we write $\theta_s$ for $\theta(\{s\})$. \\
\begin{proof}
First note an elementary fact:
\begin{equation}
\mbox{if $0<x<1$ and $\beta \ge 0$ and $x \ge x^2 + 2x(1-x)\beta$ then $\beta \le 1/2$.}
\label{quadr}
\end{equation}
From the construction with $(Y, Y_1, Y_2, X)$ we have
\begin{eqnarray}
\theta_0 &=& \theta_0^2 + 2 \Pr(Y_1 = 0, Y_2> 0, Y>Y_2/2) + \Pr(Y_1 = 0, Y_2> 0, Y=Y_2/2) \nonumber \\
               &=& \theta_0^2 + 2 \theta_0 \left(  \Pr(2Y>Y_2, Y_2>0)    + \sfrac{1}{2} \Pr(2Y=Y_2, Y_2>0)                \right) \nonumber \\
               &=& \theta_0^2 + 2 \theta_0 (1 - \theta_0) \left(  \Pr(2Y>Y_2 \vert Y_2>0)    + \sfrac{1}{2} \Pr(2Y=Y_2 \vert Y_2>0)  \nonumber  \right) \\
               &\ge& \theta_0^2 + 2 \theta_0 (1 - \theta_0)  \left(  \Pr(Y>1/2) + \Pr(Y=1/2, Y_2< 1 \vert Y_2>0) + \sfrac{1}{2} \Pr(Y=1/2, Y_2 =1\vert Y_2>0)  \right) \nonumber \\
               &&   \label{ttt1} \\
               & = & \theta_0^2 + 2 \theta_0 (1 - \theta_0)  \left(  \Pr(Y>1/2) + \theta_{1/2} \  [  \Pr(Y_2< 1 \vert Y_2>0) +  \sfrac{1}{2} \Pr(Y_2 =1\vert Y_2>0) ] \right) \nonumber  \\
               & = & \theta_0^2 + 2 \theta_0 (1 - \theta_0)  \left(  \Pr(Y>1/2) + \theta_{1/2} \  [ \sfrac{1}{2} + \sfrac{1}{2}  \Pr(Y_2< 1 \vert Y_2>0) ] \right) \nonumber  \\
               & \ge & \theta_0^2 + 2 \theta_0 (1 - \theta_0)  \left(  \Pr(Y>1/2) + \sfrac{1}{2} \Pr(Y = 1/2)  \right). \label{ttt2}
 \end{eqnarray}
 By hypothesis $\theta_0 > 0$, so by (\ref{quadr}) we have  $\Pr(Y>1/2) + \sfrac{1}{2} \Pr(Y = 1/2) \le 1/2$.
 However we have the analogous sequence of equalities and inequalities for $\theta_1$, which imply  $\Pr(Y<1/2) + \sfrac{1}{2} \Pr(Y = 1/2) \le 1/2$, and so we must have
 \[ \Pr(Y>1/2) + \sfrac{1}{2} \Pr(Y = 1/2) = 1/2 = \Pr(Y<1/2) + \sfrac{1}{2} \Pr(Y = 1/2) . \]
 The quantity (\ref{ttt2}) now {\em equals} $\theta_0$, so the inequalities at (\ref{ttt1}) and (\ref{ttt2}) must in fact be equalities.
 In order for  the inequality leading to (\ref{ttt2}) to be an equality, we must have either $\theta_{1/2} = 0$ or $ \Pr(Y_2< 1 \vert Y_2>0) = 0$.
 In the latter case,  $\theta$ is supported on $\{0,1\}$  and so  $\theta = \delta_{0,1}$, as desired.
 So the remaining case is $\theta_{1/2} = 0$.
 In this case, for  the inequality leading to (\ref{ttt1}) to be an equality, we must have
 $\Pr(Y_2 < 2Y, Y < 1/2 \vert Y_2 > 0) = 0$.
 But, as at the end of the proof of Lemma \ref{L:0122}, this can only happen if $\Pr(0 < Y < 1/2) = 0$.
 By the analogous argument for $\theta_1$ we have $\Pr(1/2 < Y < 1) = 0$, and so the distribution is supported on $\{0,1\}$ and must be $\delta_{0,1}$, as desired.
\qed
\end{proof}

This line of argument can be extended to some other values of $(j,k)$ -- see \cite{mappings_long}.

\section{A class of tree spaces}
\label{sec:trees}

\setlength{\unitlength}{0.1in}
\begin{figure}
\begin{picture}(30,22)(-8,0)
\put(10,10){\line(-1,0){5}}
\put(10,10){\line(1,1){5}}
\put(10,10){\line(5,-4){5}}
\put(15,6){\line(2,1){8}}

\put(5,10){\line(-3,4){3}}
\put(2,14){\circle*{1}}
\put(5,10){\line(-1,-3){2}}
\put(3,4){\circle*{1}}

\put(15,15){\line(-1,2){1.5}}
\put(13.5,18){\circle*{1}}
\put(15,15){\line(1,0){4}}
\put(19,15){\circle*{1}}

\put(15,6){\line(0,-1){4}}
\put(15,2){\circle*{1}}

\put(23,10){\line(1,1){5}}
\put(28,15){\circle*{1}}
\put(23,10){\line(1,-1){7}}
\put(30,3){\circle*{1}}
\end{picture}
\caption{A BTL space $S$ with $|S| = 7$.}
\label{Fig:BTL}
\end{figure}
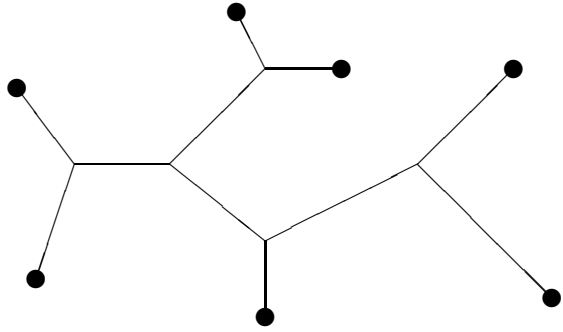

In this section we consider  binary\footnote{Essentially the same argument works without the {\em binary} assumption.} tree leaves (BTL), illustrated in Figure \ref{Fig:BTL}, as a class of finite spaces.
Here $S$ is the finite set of leaves; 
the edges have lengths which serve  to determine the distance between two leaves  as the length of the unique path 
between them; the edges also define $|S| - 2$ branchpoints.
To ``break symmetry" we assume
\begin{equation}
\mbox{ all distances $(d(s_i,s_j), j \ne i)$ are distinct.}
\label{distinct}
\end{equation}
We claim that, 
as suggested by the general picture from numerics, 
 for $k = 2$ there are no invariant measures other
than the omnipresent ones.  
An invariant measure supported on a subset of leaves is an invariant measure on the induced spanning tree of that subset,
so to prove that claim it suffices to prove
\begin{Theorem}
\label{Conj}
On a BTL space $S$ with $|S| \ge 3$ and satisfying \eqref{distinct}, and for $k=2$,
there are no invariant measures with full support.
\end{Theorem}
\begin{proof}
As in previous proofs,
consider a $\pi_{1,2}$ - invariant distribution $\theta$ with full support on $S$, where $|S| \ge 3$.
Take $Y_0, Y_1, Y_2$ independent with  distribution $\theta$.
Invariance says that the random variable $X$ defined as
\begin{eqnarray}
 X &=& Y_1 \mbox{ if } d(Y_0,Y_1) < d(Y_0,Y_2) \nonumber\\
    &=& Y_2 \mbox{ if } d(Y_0, Y_2) < d(Y_0, Y_1) \label{recurse}
\end{eqnarray}    
will also have distribution $\theta$.
We proceed to a proof by contradiction.

We quote an elementary fact.
\begin{Lemma}
For any probability distribution $\theta$ on a BTL space $S$, either\\
(i) $\theta(s_0) > \frac{1}{2}$ for some $s_0 \in S$\\
or (ii) there exists a centroid, that is a branchpoint such that the associated partition $S = \cup_{i=1}^3 A_i$ 
of leaves satisfies $0 < \theta(A_i) \le \frac{1}{2}$ for all $i$.
\end{Lemma}
Consider case (i).
That is, suppose $\theta$ is invariant and $\theta(s_0)  \in (\frac{1}{2}, 1)$.
From the invariance relation \eqref{recurse}, in order that $X = s_0$ it is sufficient that
\[ (Y_0 \neq s_0, Y_1 = s_0, Y_2 = s_0) \mbox{ or } (Y_0 = s_0, Y_1 \mbox{ or } Y_2 = s_0) .\]
So, setting $\theta(s_0) = x \in (\frac{1}{2}, 1)$,
\[ x \ge (1-x)x^2 + x(2x - x^2) . \]
Cancelling $x$, this reduces to $2x^2 - 3x +1 \ge 0$, but this inequality is false for $ x \in (\frac{1}{2}, 1)$.

Now consider case (ii).
There is a centroid branchpoint defining a partition $S = \cup_{i=1}^3 A_i$.
Consider the leaf $s_1$ which is closest to the centroid.
We may assume $s_1 \in A_1$.
From the invariance relation \eqref{recurse}, in order that $X = s_1$ it is sufficient that the following condition (*) holds:
\begin{quote}
exactly one of $(Y_1,Y_2)$ equals $s_1$\\
and\\
$Y_0$ and the other\footnote{The leaf from $(Y_1,Y_2)$ that is not $s_1$.} $Y$ are in different components of $\cup_{i=1}^3 A_i$.
\end{quote}
For instance, if $Y_0 \in A_2$ and $Y_1 = s_1$ and $Y_2 \in A_3$, then
$Y_2$ is some leaf in $A_3$ which is farther from the centroid than is $s_1$, so $d(Y_0,s_1) < d(Y_0,Y_2)$.
The other possibilities are similar.

By considering the three possibilities for `` different components of $\cup_{i=1}^3 A_i$" we see that
the probability of (*) equals $\theta(s_1)$ times
\[
 2 \theta(A_1) \theta(A_2) + 2 (\theta(A_1) - \theta(s_1)) \theta(A_2) 
 \]\[
 +  2 \theta(A_1) \theta(A_3) + 2 (\theta(A_1) - \theta(s_1)) \theta(A_3) 
 \]\[
 +  4 \theta(A_2) \theta(A_3) 
\]
which rearranges to 
\[
 \theta(s_1)
[ 4(\theta(A_1) \theta(A_2) + \theta(A_1) \theta(A_3) + \theta(A_2) \theta(A_3)) 
- 2 \theta(s_1)(\theta(A_2) + \theta(A_3)) ]
 \] 
 \begin{equation}
  = \theta(s_1) \cdot B, \mbox{ say.}
  \label{ttt}
  \end{equation}
A disjoint sufficient condition for $X = s_1$ is that $Y_1 = Y_2  = s_1$, which has probability $\theta^2(s_1)$. 
So
\[ \theta(s_1) = \Pr(X = s_1) \ge \theta(s_1) (B + \theta(s_1)) . \]
Cancelling the $\theta(s_1)$ term,
\[ 1 \ge  4(\theta(A_1) \theta(A_2) + \theta(A_1) \theta(A_3) + \theta(A_2) \theta(A_3)) 
- 2 \theta(s_1)(\theta(A_2) + \theta(A_3) - \sfrac{1}{2})  . \]
Because $\sum_i \theta(A_i) = 1$ we have $\theta(A_2) + \theta(A_3) - \sfrac{1}{2} =  \sfrac{1}{2} - \theta(A_1)$
and
\[  2(\theta(A_1) \theta(A_2) + \theta(A_1) \theta(A_3) + \theta(A_2) \theta(A_3)) = 1 - \sum_i \theta^2(A_i) \]
and the inequality above reduces to 
\[ 1 \ge 2 - 2  \sum_i \theta^2(A_i) - 2 \theta(s_1) ( \sfrac{1}{2} - \theta(A_1)) . \]
Because $\theta(s_1) \le \theta(A_1)$, this implies
\begin{equation}
C : =  \sum_i \theta^2(A_i) +  \theta(A_1) ( \sfrac{1}{2} - \theta(A_1)) \ge  \sfrac{1}{2} .
 \label{tt}
 \end{equation}
We need to show that $C \ge  \sfrac{1}{2} $ cannot in fact occur under the constraints $0 < P(A_i) \le \sfrac{1}{2} $ and
$\sum_i \theta(A_i) = 1$.
Given $\theta(A_1) = x$, the quantity $C$ is maximized when $(\theta(A_2), \theta(A_3)) = ( \sfrac{1}{2},  \sfrac{1}{2} - x)$
and so 
\[ C \le x^2 + \sfrac{1}{4} + (\sfrac{1}{2} - x)^2 + x (\sfrac{1}{2} - x) 
= x^2 - \sfrac{1}{2} x + \sfrac{1}{2}.
\]
This implies that $C <  \sfrac{1}{2}$ on the open interval $x \in (0,\sfrac{1}{2})$, and we cannot have $x = 0$ or $\sfrac{1}{2}$
by the $\theta(A_i) > 0$ constraint.
\end{proof}
\qed


\end{document}